\documentclass{amsart}

\numberwithin{equation}{section}

\usepackage{color} 
\usepackage{amstext} \usepackage{amsthm} \usepackage{amsmath} \usepackage{amssymb}
\usepackage{latexsym} \usepackage{amsfonts} \usepackage{graphicx} \usepackage{texdraw} \usepackage{graphpap}
\usepackage{enumerate}

\usepackage{amsrefs}

\usepackage[inline,nomargin]{fixme}

\newcommand{\bbr}{\mathbb{R}}
\newcommand{\bbc}{\mathbb{C}}
\newcommand{\bbn}{\mathbb{N}}
\newcommand{\bbz}{\mathbb{Z}}
\newcommand{\zo}{\mathbb{Z}_0}

\newcommand{\bomega}{{\partial \Omega}}
\newcommand{\pa}{\partial}

\newcommand{\calh}{\mathcal{H}}

\newcommand{\grad}{\nabla}
\newcommand{\dudn}{{\dfrac{\pa u}{\pa \nu}}}
\newcommand{\dwdn}{{\dfrac{\pa w}{\pa \nu}}}
\newcommand{\lbot}{{L^2(\bomega \times [0,T])}}

\DeclareMathOperator{\sgn}{sgn}
\DeclareMathOperator{\re}{Re}

\newtheorem{theorem}{Theorem}[section]
\newtheorem{lemma}[theorem]{Lemma}
\newtheorem{proposition}[theorem]{Proposition}

\theoremstyle{definition}
\newtheorem{definition}[theorem]{Definition}

\theoremstyle{remark}

\title[Harm. Pf. of Observability for Wave and Visco-Elastic Eq.]{A Harmonic Analysis Proof of the Boundary Observability Inequality for the Wave Equation and Visco-Elastic Equation}
\author[W. Green]{Walton Green}
\address{Department of Mathematical Sciences, Clemson University, Clemson, South Carolina 29634}
\email{awgreen@clemson.edu, liul@clemson.edu, mmitkov@clemson.edu}
\author[S. Liu]{Shitao Liu}
\author[M. Mitkovski]{Mishko Mitkovski}



\begin{document}
\maketitle

\begin{abstract}
In this paper, we give a harmonic analysis proof of the Neumann boundary observability inequality for the wave equation in an arbitrary space dimension. Our proof is elementary in nature and gives a simple, explicit constant. 
We also extend the method to prove the observability inequality of a visco-elastic wave equation.
\end{abstract}

\section{Introduction}

Let $\Omega \subseteq \bbr^d$ be an open, bounded domain with smooth boundary $\bomega$. Consider the following backwards wave equation generated at final time $T$.
\begin{equation} \label{waveeq}\left\{ \begin{array}{lr}
	w_{tt}(x,t)-\Delta w(x,t) =0 & \mbox{in }\Omega \times [0,T]\\
	w(x,T)=w_0(x) \quad w_t(x,T)=w_1(x) \hspace{5ex}& \mbox{in }\Omega \\
	w(x,t) = 0 & \mbox{on }\bomega \times [0,T]\end{array} \right.
\end{equation}
where $w_0 \in H^1_0(\Omega)$ and $w_1 \in L^2(\Omega)$. The problem we are interested in is the boundary observability inequality: There exists $c >0$ such that for all $(w_0,w_1) \in H^1_0(\Omega) \times L^2(\Omega)$,
	\begin{equation} \label{waveob} c \int_\Omega |\nabla w_0(x)|^2 + |w_1(x)|^2 \, dx \le \int_0^T \int_\bomega \left|\dfrac{\partial w}{\partial \nu} (x,t)\right|^2\, dS(x) \, dt \end{equation}
It is well known (e.g. by the Hilbert Uniqueness Method \cite{lions88}) that the observability inequality (\ref{waveob}) is equivalent to the exact controllability of the dual equation to (\ref{waveeq}):
\begin{equation} \label{waveeqcontrol}\left\{ \begin{array}{lr}
	u_{tt}(x,t)-\Delta u(x,t) =0 & \mbox{in }\Omega \times [0,T]\\
	u(x,0)=u_0(x) \quad u_t(x,0)=u_1(x) \hspace{5ex}& \mbox{in }\Omega \\
	u(x,t) = f(x,t) & \mbox{on }\bomega \times [0,T] 
	\end{array} \right.
\end{equation}
for $f \in L^2(\bomega \times[0,T])$ and $(u_0,u_1) \in L^2(\Omega) \times H^{-1}(\Omega)$. Exact controllability refers to the question: Given initial states $(u_0,u_1)$ and final states $(u_T,u_T')$, does there exist $T>0$ and $f \in L^2(\bomega \times[0,T])$ such that $u(x,T) = u_T(x)$ and $u_t(x,T)=u_T'(x)$?

There exists an extensive body of literature about the exact boundary controllability for the wave equation (or other typed hyperbolic equations). The problem has been very well studied and we refer to the books \cite{komornik95,L-Tbook,Lionsbook,micu02}, and the references therein for a literature review of the problem. The typical method is to use the duality and transfer the controllability problem for the wave equation into the observability question of the dual problem.

Even though such equivalence between the controllability and observability was long noticed \cite{dolecki77}, it is not until the mid 80s that mathematicians started to develop systematic methods to prove the observability inequality, especially in the general multidimensional setting. It is well known by now that the observability inequality (\ref{waveob}) may be proved using microlocal analysis \cite{b-l-r}, the multiplier method \cite{ho86,lions88}, or Carleman estimates. The latter is especially powerful and has become a major tool to prove observability inequalities since it can deal with the situations where there are lower-order terms or variable principle coefficients appearing in the wave equation. We cite only a few references here \cite{gllt,ltz,zhang10} and refer to the other works quoted in these papers for readers who are interested in the details of this method.

The idea of using harmonic analysis to prove the observability inequality originated much earlier. It seems that Russell~\cite{russell78} was the first person who systematically explored the relationship between control problems and harmonic analysis. The moment method of Russell has been extended in different directions~\cite{hansen,komornik05,loreti12,pandolfi}, but the common feature of all these results has been the requirement for the space dimension to be equal to one. Probably, the most comprehensive treatment, to date, on the use of complex exponentials and harmonic analysis in control problems is the monograph~\cite{avdonin95} where, in addition, approximate controllability results (even in higher space dimension) are obtained using complex exponentials in concert with standard uniqueness results. Still, to the best of our knowledge, nobody has given a harmonic analysis proof of the exact controllability (and observability) of the wave equation in higher space dimensions. In this note, we complete this gap by providing a proof of the observability inequality in an arbitrary spatial dimension using the harmonic analysis method. 

We also apply this method to show observability for a wave equation with memory kernel, also known as the visco-elastic wave equation, which is of the form
	\begin{equation}\label{wavemem} y_{tt}-\Delta y = \int_0^t M(t-s) \Delta y(s) \, ds.\end{equation}
Our motivation is from \cite{loreti12,pandolfi15} in which exact controllability of $(y,y_t)$ is achieved but only for dimension $d \le 3$ using the moment method of Russell. In section \ref{visco} we extend this to an arbitrary space dimension. Carleman estimates have been applied to the heat equation with memory kernel (same as (\ref{wavemem}) but only first-order in time) \cite{fu09}, but we are not aware of the use of Carleman estimates to prove the observability for the visco-elastic equation (\ref{wavemem}). A separate noteworthy contribution to this problem is \cite{kim93} in which exact controllability is established in arbitrary dimension using the classical compactness-uniqueness argument.

The paper is organized as follows. In section \ref{harm}, we reformulate the observability inequality as a Riesz sequence property for a suitably chosen system of functions. This property is then established for the regular wave equation (\ref{waveeq}) in section \ref{main} and extended to the visco-elastic wave equation (\ref{wavemem}) in section \ref{visco}.

\section{Harmonic Analysis Reformulation}
\label{harm}
The standard one-dimensional moment method focuses on estimates concerning sequences of complex exponentials. In order to extend it to higher spatial dimensions, we need to consider the following system of functions.

Let $\{\phi_n\}_{n=1}^\infty$ be an orthonormal basis in $L^2(\Omega)$ of eigenfunctions of the Dirichlet Laplacian. In other words,
\begin{equation}\left\{\begin{array}{rclr} -\Delta \phi_n &=& \lambda_n^2 \phi_n & \mbox{in } \Omega\\[2mm]
							\phi_n&=& 0 &\mbox{on } \bomega\end{array} \right. \end{equation}
It is well known that $0 < \lambda_1 \le \lambda_2 \le \cdots$ and $\lambda_n \to \infty$. For simplicity, we set $\lambda_n = \text{sgn}(n)\lambda_{|n|}$ and
	\begin{equation}\label{psidef} \psi_n = \dfrac{1}{\lambda_n}\dfrac{\pa \phi_{|n|}}{\pa \nu} \quad \text{on }\bomega\end{equation}
for $n \in \bbz \backslash \{0\}$ (henceforth $\bbz_0$), denoting by $\nu(x)$ the outward normal vector to $\bomega$ at $x$.

This system of functions $\{\psi_n\}$ has been investigated as far back as the 1940's by Rellich \cite{rellich46}. More recently, Hassel and Tao~\cite{hassell02} reinvigorated the interest in this system from the harmonic analysis perspective. As a consequence, much more precise linear independence results about it were obtained (which extend the independence results that follow from the observability inequality)~\cite{barnett11,triggiani08,xu12}. Our treatment of the system $\{\psi_n\}$ is directly influenced by this more recent work. The following two features enter in our proof: 1) the system $\{\psi_n\}$ is linearly independent in certain small spectral windows (for close values of $\lambda_n$); 2) complex exponentials are independent for $\lambda_n$ with large enough gap. Thus, when combined together, we achieve independence without restrictions on $\lambda_n$. The notion of independence we use is stronger than the classical linear independence. It is given precisely by the notion of a Riesz-Fischer sequence (see e.g.~\cite{avdonin95}, \cite{young01}).
\begin{definition}
A sequence $\{e_n\}$ in a Hilbert space $\calh$ is said to be a Riesz-Fischer sequence if there exists a constant $c>0$ such that
	\begin{equation}\label{rfseq}
		c \sum |a_n|^2 \le \left\| \sum a_ne_n \right\|^2_\calh
	\end{equation}
for all finite sequences $\{a_n\}$.
\end{definition}

We now state the relationship between Riesz-Fischer sequences and the observability inequality. This relationship is well-known to the experts in the field, but since we were not able to find a good reference, we decided to include a short proof of it. 
\begin{proposition}\label{prop1}
The observability inequality (\ref{waveob}) holds for all $(w_0,w_1) \in H^1_0(\Omega) \times L^2(\Omega)$ if $\{\psi_ne^{i\lambda_nt}\}_{n \in \bbz_0}$ is a Riesz-Fischer sequence in $L^2(\bomega \times [0,T])$, i.e. there exists $c >0$ such that
	\begin{equation}\label{riesz} c \sum |a_n|^2 \le \int_0^T \int_\bomega \left| \sum a_n \psi_n(x) e^{i\lambda_n t}\right|^2 \, dS(x) \, dt \end{equation}
for all $\{a_n\}_{n \in \bbz_0} \in \ell^2$.
\end{proposition}
\begin{proof}Let $(w_0,w_1) \in H^1_0(\Omega) \times L^2(\Omega)$.
We will represent the solution $w$ to (\ref{waveeq}) by separation of variables. In the space variable, we expand onto $\{\phi_n\}$. There exist $\{\xi_n\}, \{\eta_n\} \in \ell^2$ such that
	\[ w_0 = \sum_{n=1}^\infty \xi_n \phi_n \quad \quad w_1 = \sum_{n=1}^\infty \eta_n \phi_n \]
Since $w_0 \in H^1_0(\Omega)$, by the orthonormality of $\{\phi_n\}$,
	\[ \int_\Omega |\nabla w_0(x)|^2 \, dx = -\int_\Omega w_0\Delta w_0 \, dx = \int_\Omega \left( \sum \xi_n \phi_n \right)\left( \sum \lambda_n^2\xi_n \phi_n \right) = \sum |\lambda_n \xi_n|^2 \]
therefore $\{ \lambda_n \xi_n \}\in \ell^2$. Set $\tilde \xi_n = \lambda_n \xi_n$. Then,
	\[ w_0 = \sum \dfrac{\tilde \xi_n}{\lambda_n} \phi_n \]
Additionally, we consider the following ordinary differential equation to account for the time variable.
	\begin{equation} \left\{ \begin{array}{lrl}
		z_{jn}''(t) + \lambda_n^2 z_{jn}(t) = 0 & t \in [0,T]&j=1,2\\
		z_{1n}(T) = 1 \quad z_{1n}'(T) = 0 &t=T &\\
		z_{2n}(T) = 0 \quad z_{2n}'(T) = -\lambda_n &t=T&
	\end{array} \right. \label{ode}
	\end{equation}
Solutions to (\ref{ode}) are of the form $z_{1n}(t) = \cos(\lambda_n(T-t))$ and $z_{2n}(t) = \sin(\lambda_n(T-t))$. Thus, we can represent $w$ solving (\ref{waveeq}) as
	\begin{equation}\label{dualrep} w(x,t) = \sum \left[ \dfrac{\tilde\xi_n}{\lambda_n}\cos(\lambda_n(T-t)) - \dfrac{\eta_n}{\lambda_n}\sin(\lambda_n(T-t)) \right]\phi_n(x) \end{equation}
Then the observability inequality (\ref{waveob}) takes the following form:
	\begin{align} \notag
	c \sum &|\tilde \xi_n|^2 + |\eta_n|^2 \\
	&\le \int_0^T \int_\bomega \left|\sum  \left[ \tilde \xi_n\cos(\lambda_n(T-t)) - \eta_n\sin(\lambda_n(T-t)) \right]\dfrac{1}{\lambda_n}\dfrac{\partial \phi_n}{\partial \nu}(x) \right|^2 \, dS(x)\, dt \label{equiv}
	\end{align}
Using the Euler formula and setting $a_n = \tilde \xi_{|n|} + i \sgn(n)\eta_{|n|}$ for $n \in \bbz_0$, (\ref{waveob}) and (\ref{equiv}) are equivalent to
	\[ c \sum |a_n|^2 \le \int_0^T \int_\bomega \left| \sum a_n \psi_n e^{i \lambda_n t} \right|^2 \, dS \, dt \]
\end{proof}

\section{Main Result}\label{main}

\begin{theorem}\label{waveriesz}
Let $R>0$ such that $\Omega \subseteq B(x_0,R)$ for some $x_0 \in \bbr^d$. Then, for $T > 2R$, there exists $c >0$ such that
	\begin{equation} c \sum |a_n|^2 \le \int_0^T \int_\bomega \left| \sum a_n \psi_n(x) e^{i \lambda_n t} \right|^2 \, dS(x) \, dt
\end{equation}
for all $\{a_n\} \in \ell^2$.
Moreover, $c = \dfrac{2(T-2R)}{C_\Omega}$ where $C_\Omega$ is a positive constant dependent only on $\Omega$.
\end{theorem}
We first state two preliminary lemmas concerning the functions $\{\psi_n\}$. Define the following operator $A:H^1_0(\Omega) \to L^2(\Omega)$ which connects the boundary terms $\psi_n$ with the interior eigenfunctions $\phi_n$.
	\begin{equation} (Au)(x) = m(x) \cdot \nabla u(x) \mbox{ where } m(x) = x-x_0 \label{am} \end{equation}
for $u \in H^{1}_0(\Omega)$ and $x \in \Omega$.

\begin{lemma}\label{quasio} Let $A$ and $m$ be defined by (\ref{am}). Then, for all $j,k \in \bbz_0$,
\begin{equation} 
\int_\bomega (m \cdot \nu) \psi_j \overline{\psi_k} \, dS = \left\{ \begin{array}{cl} \displaystyle \dfrac{\lambda_j^2-\lambda_k^2}{\lambda_j\lambda_k}\int_{\Omega}A\phi_{|j|}\overline{\phi_{|k|}} \, dx &  \text{if } |j|\neq |k|; \\[5mm] 2 & \text{if }j=k; \\[2mm] -2 &\text{if } j=-k \end{array} \right. \end{equation}
\end{lemma}

\begin{proof}
We use the fact that
\begin{equation*}
A\phi_j(x) = (m\cdot\nu)\frac{\pa \phi_j}{\pa\nu}(x), \ \forall x\in\pa\Omega, \ j \in \bbn 
\end{equation*}
as in \cite{rellich46} since $\phi_j = 0$ on $\bomega$. Applying Green's Theorem and the fact that $\Delta A - A\Delta = 2\Delta$,
\begin{align*}
\int_{\partial\Omega}(m\cdot\nu)\psi_j(x)&\overline{\psi_k(x)} \ dS 
=\frac{1}{\lambda_j\lambda_k}\int_{\partial\Omega}A\phi_{|j|}\frac{\pa\overline{\phi_{|k|}}}{\pa\nu} \ dS \\[2mm]
& = \frac{1}{\lambda_j\lambda_k}\int_{\Omega} A\phi_{|j|}\Delta\overline{\phi_{|k|}} - \Delta(A\phi_{|j|})\overline{\phi_{|k|}} \ dx\\[2mm]
& = \left\{ \begin{array}{cl}\displaystyle \frac{\lambda_j^2-\lambda_k^2}{\lambda_j\lambda_k}\int_{\Omega}A\phi_{|j|}\overline{\phi_{|k|}} \ dx &  \text{if } |j|\neq |k|; \\[3mm] \displaystyle \frac{1}{\lambda_j\lambda_k}\int_{\Omega}2\lambda_j^2|\phi_{|j|}|^2 \ dx = \pm2 & \text{if } | j|=|k|. \end{array} \right.
\end{align*}
\end{proof}

\begin{lemma}\label{abdd}
The sequence $\{\lambda_j^{-1} A\phi_{|j|}\}_{j \in\zo}$ is quasi-orthogonal in $L^2(\Omega)$. More precisely, for all $u \in \ell^2(\bbz_0)$,	\begin{equation}\label{abddeq} \int_\Omega \left| \sum_j u_j \dfrac{A \phi_{|j|}}{\lambda_j} \right|^2\le R^2  \sum_j \left(|u_j|^2 - u_j\bar u_{-j}\right) \end{equation}
Secondly,
	\begin{equation}\label{aadjoint} \int_\Omega A\phi_{|j|} \phi_{|k|} = -\int_\Omega \phi_{|j|} A \phi_{|k|} \end{equation}
for $|j| \ne |k|$.
\end{lemma}

\begin{proof}
Notice that the system $\{\lambda_j^{-1}\nabla \phi_{|j|} \}_{j \in \bbz_0}$ has some sense of orthogonality. Indeed, for each $j,k \in \bbz_0$, 
	\[ \int_\Omega \dfrac{\nabla \phi_{|j|} \cdot \nabla \overline \phi_{|k|}}{\lambda_j\lambda_k} = -\int_\Omega \dfrac{\phi_{|j|} \Delta \overline{\phi_{|k|}}}{\lambda_j\lambda_k} = \left\{ \begin{array}{rc} 0 & |j| \ne |k| \\ 1 & j=k \\ -1 & j=-k\end{array} \right. \]
Then, using the definition of $A$ in (\ref{am}) and the Cauchy-Schwarz Inequality, we obtain for $\{u_j\} \in \ell^2(\bbz_0)$
\[ \int_\Omega \left| \sum_j u_j \dfrac{A \phi_j}{\lambda_j}\right|^2 
\le R^2\int_\Omega\left|\sum_j u_j \dfrac{\nabla \phi_j}{\lambda_j} \right|^2 = R^2 \left( \sum_j |u_j|^2 - \sum_j u_j\bar u_{-j} \right) \]

Now we proceed to the second statement in the lemma. Recalling $m$ from (\ref{am}) and using $\pa_i$ to denote $\frac{\pa}{\pa x_i}$,
	\[ m_i \pa_i \phi_j \overline \phi_k = \pa_i (m_i \phi_j \overline \phi_k) - (\pa_i m_i) \phi_j \phi_k - m_i \phi_j \pa_i \overline{\phi_k} \]
Summing over $i=1,\ldots,d$ and integrating over $\Omega$ yields
	\begin{equation}\label{asym} \int_\Omega A \phi_j \overline{\phi_k} = \int_\Omega \nabla \cdot (m \phi_j\overline{\phi_k}) - d \int_\Omega \phi_j \overline{\phi_k} - \int_\Omega \phi_j A\overline{\phi_k} \end{equation}
which gives the desired identity since $\phi_j=0$ on $\bomega$ and $\{\phi_j\}$ are orthonormal.
\end{proof}

Now we complete the proof of Theorem 1. To be concise, all sums are assumed to be taken over $\bbz_0$ unless otherwise stated. For $C_{\Omega} :=\max_{x\in\pa\Omega}[m(x)\cdot\nu(x)]\le R$, we have the following estimate using Lemma \ref{quasio}.
\begin{align}
	C_{\Omega}\int_{\pa\Omega}\int_0^T&\left|\sum_{j }a_je^{i\lambda_jt}\psi_j(x)\right|^2 \ dt\, dS
	\ge \sum_{j}\sum_{k}a_j\bar a_k\int_{0}^{T}e^{i(\lambda_j-\lambda_k)t} \ dt \int_{\pa\Omega}(m\cdot\nu)\psi_j\overline{\psi_k} \ dS \nonumber\\[2mm]
	= 2T\sum_{j}&|a_j|^2 -\sum_j a_j\bar{a}_{-j} \dfrac{e^{i2\lambda_jT}-1}{i\lambda_j} \notag\\
		&+\sum_{j}\sum_{k\neq \pm j}a_j\bar{a}_k\left(\frac{1}{i\lambda_j}+\frac{1}{i\lambda_k}\right)\left(e^{i(\lambda_j-\lambda_k)T}-1\right)\int_{\Omega}A\phi_j\overline{\phi_k} \ dx \label{orig}
\end{align}
Notice when $k= j$, the terms in the double summation actually have zero value. Thus we may include them in the summation. Moreover, applying the second statement in Lemma \ref{abdd},
	\[ \sum_{j}\sum_{k\neq \pm j}a_j\bar{a}_k\left(\frac{1}{i\lambda_j}+\frac{1}{i\lambda_k}\right)\left(e^{i(\lambda_j-\lambda_k)T}-1\right)\int_{\Omega}A\phi_j\overline{\phi_k} \ dx  \] \[= 2 \re\left(\sum_{j}\sum_{k\neq - j}a_j\bar{a}_k\left(e^{i(\lambda_j-\lambda_k)T}-1\right)\int_{\Omega}\dfrac{A\phi_j}{i\lambda_j}\overline{\phi_k} \ dx\right)\] 
We will now include the terms when $k=-j$. Notice that by (\ref{asym}), $ 2\re\left(\int_\Omega A\phi_{|j|}\bar \phi_{|-j|} \right) = -d$.
Thus we can rewrite the second two terms in the original inequality (\ref{orig}) as
	\begin{align*} 
		&(d-1)\re\left(\sum_j a_j\bar a_{-j} \dfrac{e^{i2\lambda_jT}-1}{i\lambda_j}\right) \\
		&\hspace{10ex}+2 \re\left(\sum_{j}\sum_{k}a_j\bar{a}_k\left(e^{i(\lambda_j-\lambda_k)T}-1\right)\int_{\Omega}\dfrac{A\phi_j}{i\lambda_j}\overline{\phi_k} \, dx\right) 
	\end{align*}
Additionally, we have the following identity for the single sum:
	\begin{align*} &\sum_j a_j\bar a_{-j}\dfrac{e^{i2\lambda_jT}-1}{i\lambda_j} \\
	&\hspace{5ex}= \int_\Omega \left( \sum_j a_j e^{i\lambda_jT} \dfrac{\phi_{|j|}}{i\lambda_j} \right)\overline{\left( \sum_k a_k e^{i\lambda_kT} \phi_{|k|} \right)}- \left( \sum_j a_j \dfrac{\phi_{|j|}}{i\lambda_j} \right)\overline{\left( \sum_k a_k \phi_{|k|} \right)}dx\end{align*}
Then we split the double sum into two terms (one with $e^{i(\lambda_j-\lambda_k)T}$ and one with $-1$) and estimate each with the corresponding portion in the above identity.
	\begin{align}\notag &\left| \int_\Omega (d-1)\left( \sum_j a_j e^{i\lambda_jT} \dfrac{\phi_{|j|}}{i\lambda_j} \right)\overline{\left( \sum_k a_k e^{i\lambda_kT} \phi_{|k|} \right)} +2\sum_{j}\sum_{k}a_j\bar a_ke^{i(\lambda_j-\lambda_k)T}\int_{\Omega}\dfrac{A\phi_{|j|}}{i\lambda_j}\overline{\phi_{|k|}} \right|\\
	& \hspace{7ex}=\left| \int_\Omega \left((d-1)\sum_j a_j e^{i\lambda_jT} \dfrac{\phi_{|j|}}{\lambda_j} + 2\sum_j a_j e^{i\lambda_jT}\dfrac{A\phi_{|j|}}{\lambda_j}\right)\overline{\left( \sum_k a_k e^{i\lambda_kT} \phi_{|k|} \right)} \right|\notag \\
	&\hspace{7ex} \le \dfrac{1}{4R} \int_\Omega \left|(d-1)\sum_j a_j e^{i\lambda_jT} \dfrac{\phi_{|j|}}{\lambda_j} + 2\sum_j a_j e^{i\lambda_jT}\dfrac{A\phi_{|j|}}{\lambda_j}\right|^2 + R \left| \sum_k a_k e^{i\lambda_kT} \phi_{|k|} \right|^2 \label{id2}\end{align}
Note that by (\ref{asym}), for any $u \in H^1_0(\Omega)$,
	\[ \|(d-1)u+2Au\|^2 = (d-1)^2\|u\|^2 + 4(d-1)\re(u,Au) + 4\|Au\|^2  \] \[= (-1-d)(d-1)\|u\|^2 + 4\|Au\|^2 \le 4\|Au\|^2 \] 
where $\|\cdot\|$ and $(\cdot,\cdot)$ denote the $L^2(\Omega)$ norm and inner product. Apply this to (\ref{id2}) with $u=\sum a_je^{i\lambda_jT}\phi_{|j|}\lambda_j^{-1}$. Then, applying (\ref{abddeq}) from Lemma \ref{abdd}, we have
	\begin{equation} \label{id3} \dfrac{1}{R} \int_\Omega \left| \sum_j a_j e^{i\lambda_jT}\dfrac{A\phi_{|j|}}{\lambda_j}\right|^2 + R \int_\Omega \left| \sum_k a_k e^{i\lambda_kT} \phi_{|k|} \right|^2 \end{equation}
	\[ \le R \sum_j \left(|a_j|^2 - a_j\bar a_{-j} e^{i2\lambda_jT} \right) +R \sum_k \left(|a_k|^2 + a_k\bar a_{-k} e^{i2\lambda_kT} \right) = 2R \sum_j |a_j|^2 \]

The other term (with $e^{i(\lambda_j-\lambda_k)T}$ replaced by $-1$) is estimated in a manner similar to (\ref{id2}) and (\ref{id3}). Substituting (\ref{id3}) and the corresponding estimate for $-1$ into the original inequality gives the desired result:

\begin{align*}
C_{\Omega}\int_0^T&\int_{\pa\Omega}\left|\sum_{j}a_je^{i\lambda_jt}\psi_j(x)\right|^2 \, dS \, dt \\[2mm]
&\geq 2T\sum_{j}|a_j|^2 + \sum_j a_j\bar{a}_{-j} \dfrac{e^{i2\lambda_jT}-1}{i\lambda_j} \\
&\hspace{5ex}+ \sum_{j}\sum_{k\neq \pm j}a_j\bar a_k\left(\frac{1}{i\lambda_j}+\frac{1}{i\lambda_k}\right)\left(e^{i(\lambda_j-\lambda_k)T}-1\right)\int_{\Omega}A\phi_{|j|}\overline{\phi_{|k|}}  \\[2mm]
&\geq 2T\sum_{j}|a_j|^2 - 4R\sum_{j}|a_j|^2 =2\left(T - 2R\right)\sum_j|a_j|^2.
\end{align*}
\qed

\section{Application to the Visco-Elastic Equation}\label{visco}
Let $\Omega \subseteq \bbr^d$ be an open, bounded domain with smooth boundary. Consider the following visco-elastic wave equation:
\begin{equation}\left\{ 
	\begin{array}{lr}
		y_{tt}(x,t) -\Delta y(x,t)= \displaystyle\int_t^T M(s-t)\Delta y(x,s) \, ds \hspace{5ex}&\text{in } \Omega \times [0,T] \\[2.5mm]
		y(x,T) = y_0(x) \quad y_t(x,T) = y_1(x) \hspace{5ex} &\text{in } \Omega  \\[1.5mm]
		y(x,t) = 0 &\text{on } \bomega \times [0,T]
	\end{array} \right. \label{viscodual}
\end{equation}
for given $M \in H^2(0,T)$, $y_0 \in H^1_0(\Omega)$ and $y_1 \in L^2(\Omega)$. By a similar argument to the one presented section \ref{harm}, the boundary observability inequality (analogous to (\ref{waveob})) for (\ref{viscodual}) will be acheived if it can be established that $\{z_n(t)\psi_n(x)\} \subseteq \lbot$ is a Riesz-Fisher sequence where $\psi_n$ is as defined in (\ref{psidef}) and $z_n$ satisfies the following time ODE (compare with (\ref{ode})) for each $n \in \zo$.
		\begin{equation} \left\{ \begin{array}{lr}
		z_n''(t) + \lambda_n^2 z_n(t) = -\lambda_n^2 \displaystyle\int_t^T M(s-t) z_n(s) \, ds & t \in [0,T]\\[3mm]
		z_{n}(T) = 1 \quad \quad z_{n}'(T) = i\lambda_n \\[2mm]
		\end{array} \right. \label{viscoode}
		\end{equation}
This formulation is thoroughly carried out in \cite{loreti12} and \cite{pandolfi15} where the following Riesz sequence property (\ref{viscors}) is obtained for space dimension $d=1$ and $d \le 3$ respectively. Here we extend this to the general case $d \ge 1$.

\begin{theorem}\label{viscoriesz}
Let $R>0$ such that $\Omega \subseteq B(x_0,R)$ for some $x_0 \in \bbr^d$. If $T>2R$, then for $\{z_n\}$ solving (\ref{viscoode}) and $\{\psi_n\}$ as defined in (\ref{psidef}), $\{z_n\psi_n\}$ is a Riesz sequence in $\lbot$. In other words, there exists $c,C>0$ such that
	\begin{equation}\label{viscors} c\sum |a_n|^2 \le \int_0^T \int_\bomega \left| \sum a_n z_n(t) \psi_n(x) \right|^2 \, dS(x) \, dt \le C \sum |a_n|^2 \end{equation}
for all finite sequences $\{a_n\}$.
\end{theorem}

The notion of a Riesz sequence is slightly stronger than that of a Riesz-Fisher sequence (\ref{rfseq}), we simply add the upper inequality. Nonetheless, the lower inequality is enough to imply observability of (\ref{viscodual}) which gives exact controllability for the dual system.

Our approach is similar to \cite{loreti12} and \cite{pandolfi15} in the sense that we will argue that $\{z_n\psi_n\}$ is in a certain sense ``close'' to $\{e^{i\lambda_n t} \psi_n\}$ which we already know to be a Riesz-Fisher sequence (Theorem \ref{waveriesz}). In \cite{loreti12}, it is shown that there exists $C_1>0$ such that
	\begin{equation}\label{zest} \int_0^T |z_n(t) - e^{(\gamma + i \lambda_n) t}|^2 \, dt \le \dfrac{C_1}{\lambda_n^2} \quad \forall \, t \in [0,T] \end{equation}
for some $\gamma \in \bbc$ in the special case where $\lambda_n=n$. However, there is no crucial role played by $n$ in the computations so (\ref{zest}) can be easily verified with general $\lambda_n$. The key in \cite{loreti12} is that when $\lambda_n=n$, $\{z_n\}$ and $\{e^{\gamma + i \lambda_n t} \}$ are quadratically close, which means
	\[ \sum_n \int_0^T |z_n(t)-e^{(\gamma + i\lambda_n) t} |^2 \, dt < \infty \]
In \cite{pandolfi15}, the decay (\ref{zest}) is improved to $\lambda_n^{-4}$ so quadratically closeness follows from Weyl's lemma when $d \le 3$. We do not expect to be able to extend the quadratically close property to arbitrary dimensions. Rather, we incorporate the estimates on $\{\psi_n\}$ given below in Lemma \ref{psib} to show that $\{z_n\psi_n\}$ and $\{e^{(\gamma + i \lambda_n) t} \psi_n\}$ are Paley-Weiner close. The Theorem~\ref{viscoriesz} will then become a consequence of the following variation of the classical Paley-Weiner theorem \cite{young01}:
\begin{lemma}\label{perr}
Let $\{e_n\}$ be a Riesz sequence in a Hilbert space $\calh$ and $\{f_n\} \subseteq \calh$ be an $\ell^2$-independent sequence, i.e. any $\{c_n\} \in \ell^2$ such that $\sum c_nf_n=0$ implies $c_n=0$ for all $n$. If there exists $q \in (0,1)$ and a finite set of indices $J$ such that
	\begin{equation}\label{percond} \left\| \sum_{n \not\in J} a_{n}(e_{n}-f_{n}) \right\|^2 \le q \left\| \sum_{n} a_{n}e_{n} \right\|^2 \end{equation}
for all finite sequences $\{a_{n}\}$, then $\{f_n\}$ is also a Riesz sequence.
\end{lemma}
Thus the Theorem \ref{viscoriesz} will be established if we can show three conditions hold:
	\begin{itemize}
		\item[(i)] $\{e^{(\gamma + i \lambda_n) t}\psi_n\}$ is a Riesz sequence.
		\item[(ii)] $\{z_n\psi_n\}$ is an $\ell^2$-independent sequence.
		\item[(iii)] There exists $q \in (0,1)$ and a finite set of indices $J$ such that,
			\[ \left\| \sum_{n \not\in J} a_n \psi_n \left(z_n - e^{(\gamma + i\lambda_n) t}\right) \right\|^2 \le q \left\| \sum_n a_n \psi_n e^{(\gamma + i\lambda_n) t} \right\|^2 \]
		for all finite sequences $\{a_n\}$ ( Here and henceforth $\|\cdot\|$ denotes the $\lbot$ norm).
	\end{itemize}

\begin{proof}[Proof of (i)]
We first claim that when $T>2R$, $\{\psi_ne^{i\lambda_nt}\}$ is actually a Riesz sequence, i.e. in addition to the lower inequality from Theorem \ref{main}, there exists $C_2>0$ such that
	\[ \int_0^T \int_\bomega \left| \sum a_n \psi_n(x) e^{i \lambda_nt} \right|^2 \, dS \, dt \le C_2 \sum |a_n|^2 \]
for all finite sets of scalars $\{a_n\}$. This can be established in a similar manner to Theorem \ref{waveriesz} but with the operator $A$ (\ref{am}) replaced by $V$ from the proof of Lemma \ref{psib} below. Alternatively, in the setting of Proposition \ref{prop1}, it is equivalent to the following regularity estimate for $w$ solving (\ref{waveeq}) which is well-known \cite{lasiecka86}:
	\[  \int_0^T \int_\bomega \left| \dwdn(x,t) \right|^2 \, dS \, dt \le C_2 \int_\Omega |\nabla w_0(x)|^2 + |w_1(x)|^2 \, dx\]
This is then extended to $\{e^{(\gamma + i \lambda_n) t}\psi_n\}$ by noticing that
	\begin{align} \label{4.6}
	\max\{1,e^{\re(\gamma)T}\} \left\| \sum a_n \psi_n e^{i\lambda_n t} \right\|^2 &\ge \left\| \sum a_n \psi_n e^{(\gamma + i\lambda_n) t} \right\|^2  \\ \nonumber
		&\ge \min\{1,e^{\re(\gamma)T}\} \left\| \sum a_n \psi_n e^{i\lambda_n t} \right\|^2 \end{align}
		for all finite sequences $\{a_n\}$.
\end{proof}

\begin{proof}[Proof of (ii)]
Consider a solution $y=\sum c_nz_n\phi_n\lambda_n^{-1}$ to the equation (\ref{viscodual}). In \cite{kim93}, the following unique continuation property is shown: 
\begin{quote}
Let $y$ be a solution to (\ref{viscodual}) such that 
	\[\frac{\pa y}{\pa \nu} =0 \quad \mbox{on} \quad \bomega \times [0,T] \]
If $T>2R$, then $y=0$ on $\Omega \times [0,T]$.
\end{quote}
This, in turn, gives $y(x,T)= y_t(x,T) = 0$ for $x \in \Omega$ so
	\[ 0 = \int_\Omega |\nabla y(x,T)|^2 + |y_t(x,T)|^2\, dx =\left(\sum |c_n|^2-c_n\bar c_{-n}\right) +  \left(\sum |c_n|^2+c_n\bar c_{-n}\right)\]
Therefore $c_n=0$ for all $n$.
\end{proof}

We now give the key lemma in establishing (iii).
\begin{lemma} \label{psib}Let $\{\psi_n\}$ be defined as in (\ref{psidef}). Then there exists $C_\alpha$ dependent only the domain $\Omega$ such that for any finite sequence of scalars $\{a_n\}$,
	\begin{equation}\label{psib2} \int_\bomega \left| \sum a_n \psi_n(x) \right|^2 \, dS(x) \le C_\alpha \left(\sum |a_n|^2 \right)^{1/2}\left(\sum |\lambda_na_n|^2 \right)^{1/2} \end{equation}
\end{lemma}
The estimate (\ref{psib2}) may be viewed as stating some degree of orthogonality for $\{\psi_n\}$. In proving this, we follow the techniques in \cite{barnett11,xu12} utilizing the following lemma.
\begin{lemma}[\cite{barnett11} Appendix A]
Let $\Omega$ be bounded with piecewise smooth boundary. Then, there exists a smooth vector field $\alpha$, defined on a neighborhood of $\overline{\Omega}$
such that
	\[ \alpha(x) \cdot \nu(x) \ge 1 \]
for almost every $x \in \bomega$.
\end{lemma}

\begin{proof}[Proof of Lemma \ref{psib}]
Define $V:H^1_0(\Omega) \to L^2(\Omega)$ by $(Vu)(x) = \alpha(x) \cdot \grad u(x)$.
First we claim that there exists $C_\alpha>0$ such that
	\[ \left| \int_\Omega u[V, \Delta]\bar u \, dx \right| \le C_\alpha \|\grad u\|^2 \]
for any $u \in H^3(\Omega) \cap H^1_0(\Omega)$. Indeed, using Einstein notation summing over $i,j=1,2,\ldots,d$
	\begin{align*}
		\Delta Vu &= \partial_{ii} \left( \alpha_j (\partial_j u) \right) \\
			&= (\partial_{ii}\alpha_j)(\partial_ju) + 2 (\partial_i\alpha_j)(\partial_{ij} u) + \alpha_j (\partial_{jii}u) \\
			&= V \Delta u + (\partial_{ii}\alpha_j)(\partial_ju) + 2 (\partial_i\alpha_j)(\partial_{ij} u)
	\end{align*}
	Integrating by parts once and applying the Poincar\'e inequality yields
	\begin{align*} \left| \int_\Omega uV\Delta \bar u- u \Delta V \bar u \, dx \right|&= \left|\int_\Omega u(\partial_{ii}\alpha_j)(\partial_j\bar u) + 2 u(\partial_i\alpha_j)(\partial_{ij} \bar u) \, dx \right|\\
		&= \left| \int_\Omega u(\partial_{ii}\alpha_j)(\partial_j\bar u) - 2 \left[(\partial_i \bar u)(\partial_i\alpha_j)+ u (\partial_{ii}\alpha_j)\right](\partial_{j} \bar u) \, dx \right|\\
		&\le C_\alpha \int_\Omega | \grad u |^2\, dx 
	\end{align*}
Take $u = \sum a_n\phi_n\lambda_n^{-1}$ for a finite set of scalars $\{a_n\}$. Notice that $\|\nabla u\|^2 \le 2\sum |a_n|^2$ and $\|\Delta u\|^2 \le 2\sum |\lambda_na_n|^2$ (the factor of 2 is due to the negative indices as in Lemma \ref{abdd}). Then, using Cauchy-Schwartz and the above estimates on $V$, we have
	\begin{align*}
		\int_\bomega \left|\dudn\right|^2 \, dS &\le \int_\bomega (\alpha \cdot \nu) \left|\dudn\right|^2 \, dS = \int_\bomega \dudn V\bar u \, dS \\
			&= \int_\Omega \Delta u V\bar u- u \Delta V\bar u  \, dx \\
			&=  \int_\Omega \Delta uV\bar u - uV\Delta \bar u + u[V,\Delta] \bar u \, dx \\
			&= \int_\Omega  \Delta uV\bar u +(\nabla \cdot \alpha) u\Delta \bar u + Vu \Delta \bar u + u[V,\Delta]\bar u \, dx \\
			&\le C_\alpha \left(\sum |a_n|^2 \right)^{1/2} \left( \sum |\lambda_na_n|^2 \right)^{1/2}
	\end{align*}
\end{proof}

\begin{proof}[Proof of (iii)]
Let $c_\gamma=(T-2R)\min\{1,e^{\re(\gamma)T}\}/C_\Omega$ be the constant from the lower Riesz sequence inequality (\ref{4.6}) for $\{e^{(\gamma + i\lambda_n)t}\psi_n\}$. Since $\lambda_n \to \infty$, there exists $k \in \bbn$ such that
	\[\dfrac{c_\gamma^{-1} C_\alpha C_1}{\lambda_k} < 1 \]
Take $J = \{ j : |j| < k \}$. Applying Lemma \ref{psib} and then the estimate (\ref{zest}), we have
	\begin{align*}
		\int_0^T \int_\bomega &\left| \sum_{|n| \ge k} a_n \psi_n(x) \left(z_n(t)-e^{(\gamma+i\lambda_n )t}\right) \right|^2 \\
			&\le C_\alpha \left( \int_0^T \sum_{|n| \ge k} |a_n(z_n-e^{(\gamma+i\lambda_n )t})|^2 \right)^{1/2} \left( \int_0^T\sum_{|n| \ge k} |\lambda_na_n(z_n-e^{(\gamma+i\lambda_n )t})|^2 \right)^{1/2} \\
			&\le C_\alpha C_1 \lambda_k^{-1} \sum |a_n|^2 \\
			&\le C_\alpha C_1 c_\gamma^{-1} \lambda_k^{-1} \left\| \sum a_n e^{(\gamma+i\lambda_n )t} \psi_n \right\|^2
	\end{align*}
\end{proof}

\end{document}